\documentclass[12 pt]{article}
\usepackage{amsmath, amsfonts, amsthm, latexsym}
\usepackage[all]{xy}
\usepackage{amssymb}
\usepackage{soul}
\usepackage{graphicx, color}
\usepackage[active]{srcltx}
\allowdisplaybreaks[4]  

\providecommand{\U}[1]{\protect\rule{.1in}{.1in}}
\newtheorem{theorem}{Theorem}[section]
\newtheorem{proposition}[theorem]{Proposition}
\newtheorem{corollary}[theorem]{Corollary}
\newtheorem{example}[theorem]{Example}

\newtheorem{remark}[theorem]{Remark}

\newtheorem{final remark}[theorem]{Final Remark}
\newtheorem{definition}[theorem]{Definition}
\begin{document}

\title{\sc type and cotype of multilinear operators}
\date{}
\author{Geraldo Botelho\thanks{Supported by CNPq Grant
302177/2011-6 and Fapemig Grant PPM-00326-13.}~ and Jamilson R. Campos\thanks{Supported
by a CAPES Postdoctoral scholarship.\thinspace \hfill\newline\indent2010 Mathematics Subject
Classification: 47L22, 46G25, 47H60, 46B20.\newline\indent Keywords: multi-ideals, operator ideals, type, cotype, Aron-Berner extensions.}}

\maketitle

\begin{abstract} The aim of this paper is to start the study of multilinear generalizations of the classical ideals of linear operators of type $p$ and cotype $q$. As a first step in a theory we believe will be long and fruitful, we propose a notion of type and cotype of multilinear operators and the resulting classes of such mappings are studied in the setting of the theory of Banach/quasi-Banach ideals of multilinear operators. Distinctions between the linear and the multilinear theories are pointed out, typical multilinear features of the theory are emphasized and many illustrative examples are provided. The classes we introduce are related to the multi-ideals generated by the linear ideals of operators of some type/cotype and are proved to be maximal and Aron-Berner stable.
\end{abstract}

\section*{Introduction}

\hspace{0,6cm}The concepts of type and cotype of Banach spaces and linear operators emerged in the early 1970s, mainly from the works of Dubinsky, Hoffmann-J{\o}rgensen, Maurey, Pe{\l}czy\' nski, Pisier and Rosenthal, among others, initially in the study of vector-valued independent random variables. Rapidly the theory of type and cotype became a central part of several branches of Banach space theory, such as local theory, geometry of Banach spaces, probability in Banach spaces, operator ideals, etc. The impact of the notions of type and cotype is very well described in the survey \cite{handbook} by B. Maurey, and we also mention the books by Pisier \cite{pisier}, Defant and Floret \cite{defantfloret}, Diestel, Jarchow and Tonge \cite{djt} and Pietsch \cite{pietsch}.

Pietsch \cite{pietsch83} initiated the study of ideals of multilinear operators, and since then much research has been done in this direction. The multilinear theory is far from being a simple transposition of the linear theory, good examples of the diversity of the multilinear theory can be found, e.g., in \cite{note, quaestiones, david}.

Usually, there are many possible extensions of a given operator ideal to the multilinear setting. The case of the ideal of $p$-summing linear operators is typical and several different classes of {\it absolutely summing multilinear operators} have been studied (see, e.g., \cite{quaestiones, david}). Each of these multilinear generalizations of a given operator ideal enjoy certain desired properties, and, as a rule, none of them enjoy all the expected properties. This justifies the study of several different multilinear generalizations of a given operator ideal. Each of these multilinear generalizations is supposed not only to inherit good linear properties of the original operator ideal but to enjoy good typical multilinear properties as well, such as the ones related to holomorphy types, coherent/compatible sequences and Aron-Bener stability.

In this paper we start the study of multilinear generalizations of the ideals $\tau_p$ of linear operators of type $p$ and ${\cal C}_q$ of linear operators of cotype $q$. Besides of considering the multilinear ideals generated by  $\tau_p$ and ${\cal C}_q$ by the usual methods of generating multilinear ideals from operator ideals, we adapt the linear definitions to the multilinear case and study the resulting classes. Again, this adaptation can be done in several different ways. Having in mind the history of the study of classes of absolutely summing multilinear operators, as a first step we chose the adaptation by means of the diagonal instead of the whole matrix (cf. Definition \ref{defp}). Of course other adaptations can be done (and should be done); for example, an adaptation by means of the whole matrix will lead to the concept of {\it mulilinear operators of multiple type/cotype}. The core of this paper is to investigate the concepts introduced in Definition \ref{defp} and to compare them with the classes generated by the usual methods. We pay atention to the distinctions between the linear and multilinear theories, and several typical multilinear features of the theory are pointed out along the paper. For example, in Proposition \ref{postofinito} we establish that, contrary to the linear case, scalar-valued multilinear operators do not have all possible cotypes.

In Section 2 we prove the basic properties of the ideals of multilinear operators of some type/cotype considered in this paper and we give plenty of examples/counterexamples showing that the resulting classes of multilinear operators deserve to be studied. In Section 3 we explore the relationships between the classes we have just introduced and the ideals of multilinear operators generated by the ideals of linear operators of some type/cotype. In Section 4 we prove the maximality and the Aron-Berner stability of the classes of multilinear operators of some type/cotype.

\section{Background and notation}\label{back}

\hspace{0,6cm}
Let $n \in \mathbb{N}$, $E,E_1,\ldots,E_n$ and $F$ be Banach spaces over $\mathbb{K} = \mathbb{R}$ or $\mathbb{C}$. The Banach spaces of all continuous linear operators from $E$ to $F$ and of all continuous $n$-linear operators from $E_1 \times \cdots \times E_n$ to $F$ will be represented by $\mathcal{L}(E;F)$ and $\mathcal{L}(E_1,\ldots,E_n;F)$, respectively, and by $E'$  and $\mathcal{L}(E_1,\ldots,E_n)$ if $F = \mathbb{K}$. The identity operator on $E$ is denoted by $id_E$.

For $0 < p < \infty$ we denote by $\ell_p(E)$ the $p$-Banach space (Banach space if $p \geq 1$) of absolutely $p$-summable $E$-valued sequences endowed with its usual $\ell_p$-norm $\|\cdot\|_p$; and $\ell_\infty(E)$ stands for the Banach space of bounded $E$-valued sequences with the sup norm $\|\cdot\|_\infty$.

Let $(r_j)_{j=1}^\infty$ be the Rademacher functions on $[0,1]$ (see, e.g., \cite[p.\,10]{djt}). A sequence $(x_j)_{j=1}^\infty$ in $E$ is said to be \emph{almost unconditionally summable} if the series $\sum_j r_j(t)x_j$ converges in $L_p([0,1];E)$ for some (and then all) $0 <p< \infty$. The set of these sequences, denoted by $\mathrm{Rad}(E)$, is a Banach space equipped with the norm
$$\|(x_i)_{i=1}^\infty\|_{\mathrm{Rad}(E)} = \left\Vert \sum_{j=1}^\infty r_jx_{j}\right\Vert_{L_2(E)}=\left(\int_0^1 \left\Vert \sum_{j=1}^\infty r_j(t)x_{j}\right\Vert^2 dt \right)^{1/2}.$$
We also work with the Banach space
$$\displaystyle{{\rm RAD}(E)} = \left\{(x_j)_{j=1}^\infty \in E^{\mathbb{N}}: \|(x_j)_{j=1}^\infty\|_{{\rm RAD}(E)}:= \sup_k \|(x_j)_{j=1}^k\|_{{\rm Rad}(E)}< +\infty\right\},$$
see, e.g., \cite{blascoetal,botelhocampos}. It is well known that ${\rm Rad}(E) \subseteq {\rm RAD}(E)$ and that equality holds (isometrically) if and only if $E$ does not contain a copy of $c_0$.

Given linear functionals $\varphi_1' \in E_1', \ldots, \varphi_n' \in E_n'$ and a vector $b \in F$, consider the following continuous $n$-linear operator:
$$\varphi_1 \otimes \cdots \otimes\varphi_n \otimes b \colon E_1 \times \cdots \times E_n \longrightarrow F~,~\varphi_1 \otimes \cdots \otimes\varphi_n \otimes b (x_1, \ldots, x_n) = \varphi_1(x_1) \cdots \varphi_n(x_n)b.  $$
Linear combinations of operators of this type are called {\it $n$-linear operators of finite type}. A multilinear operator has {\it finite rank} if its range generates a finite-dimensional subspace of the target space.

\begin{definition}\rm Let $n\in \mathbb{N}$ and $0 < q \leq 1$. A  {\it $q$-Banach ideal of $n$-linear operators}, or simply a {\it $q$-Banach $n$-ideal}, is a pair $({\cal M}_n, \|\cdot\|_{{\cal M}_n})$ where ${\cal M}_n$ is a subclass of the class of $n$-linear operators between Banach spaces, $\|\cdot\|_{{\cal M}_n} \colon {\cal M}_n\longrightarrow \mathbb{R}$ is a function such that, for all Banach spaces $E_1, \ldots, E_n,F$, the component
${\cal M}_n(E_1, \ldots, E_n;F) := {\cal L}(E_1, \ldots, E_n;F)\cap {\cal M}_n$
is a linear subspace of ${\cal L}(E_1, \ldots, E_n;F)$ on which $\|\cdot\|_{{\cal M}}$ is a complete $q$-norm; and:\\
(i) ${\cal M}_n(E_1, \ldots, E_n;F)$ contains all $n$-linear operators of finite type and $$\left\|I_n \colon \mathbb{K}^n \longrightarrow \mathbb{K}~,~I_n(\lambda_1, \ldots, \lambda_n) = \lambda_1 \cdots \lambda_n \right\|_{{\cal M}} =1; $$
(ii) If $A \in {\cal M}_n(E_1, \ldots, E_n;F)$, $u_m \in {\cal L}(G_m,E_m)$, $m = 1, \ldots, n$, and $t \in {\cal L}(F;H)$, then $t \circ A \circ (u_1, \ldots, u_n) \in {\cal M}_n(G_1, \ldots, G_n;H)$ and
$$\|t \circ A \circ (u_1, \ldots, u_n)\|_{{\cal M}} \leq \|t\| \cdot \|A\|_{{\cal M}} \cdot \|u_1\| \cdots \|u_n\|.  $$
\end{definition}

For $q = 1$ we say that $\cal M$ is a {\it Banach ideal of $n$-linear operators}. The case $n=1$ recovers the theory of Banach/quasi-Banach operator ideals, for which we refer to \cite{defantfloret, pietsch}. For the multilinear case see, e.g., \cite{botelhocampos, scand, prims, floret02, galicerjmaa}.

\section{Type and cotype of multilinear operators}

\hspace{0,6cm} Much research has been done on multilinear generalizations of well established operator ideals. We propose the following concepts as a first step in the extension of the ideals of operators of type and cotype to the multilinear setting.

\begin{definition}\label{defp}\rm
Let $n \in \mathbb{N}$ and  $0 <p_1,\ldots,p_n,q <\infty$ be given. We say that a continuous $n$-linear operator $T \in \mathcal{L}(E_1,\ldots,E_n;F)$ has:\\
$\bullet$ \emph{type} $(p_1,\ldots,p_n)$ if there is a constant $C>0$ such that, however we choose finitely many vectors $(x_j^{(1)},\ldots,x_j^{(n)})$ in $E_1\times \cdots \times E_n$, $j \in \{1,\ldots,k\}$,
\begin{equation}\label{multip}
   \left\|\left(T(x_j^{(1)},\ldots,x_j^{(n)})\right)_{j=1}^k \right\|_{{\rm Rad}(F) } \leq C \cdot \prod_{i=1}^n\left\Vert (x_j^{(i)})_{j=1}^k\right\Vert_{p_i} .
\end{equation}
The infimum of such constants $C$ is denoted by $\|T\|_{\tau^n_{(p_1,\ldots,p_n)}}$.\\
$\bullet$ \emph{cotype} $q$ if there exists a constant $M>0$ such that, for any choice of finitely many vectors $(x_j^{(1)},\ldots,x_j^{(n)})$ in $E_1\times \cdots \times E_n$, $j \in \{1,\ldots,k\}$,
\begin{equation}\label{multiq}
\left\|\left(T(x_j^{(1)},\ldots,x_j^{(n)})\right)_{j=1}^k \right\|_{q}  \leq M \cdot \prod_{i=1}^n \left\|(x_j^{(i)})_{j=1}^k\right\|_{\mathrm{Rad}(E)}.
\end{equation}
The infimum of such constants $M$ is denoted by $\|T\|_{C^n_{q}}$.\\
\indent The sets of all $n$-linear operators of type $(p_1,\ldots,p_n)$ and cotype $q$ from $E_1 \times \cdots \times E_n$ to $F$ are denoted by $\tau_{p_1,\ldots,p_n}(E_{1},\ldots,E_{n};F)$ and ${\cal C}^n_{q}(E_{1},\ldots,E_{n};F)$, respectively. \end{definition}

Making $n=1$ we recover the well studied Banach ideals of linear operators of type $p$ and of cotype $q$. Remember that a normed space $E$ has type $p$ (cotype $q$) if $id_E$ has type $p$ (cotype $q$). For background on type and cotype of spaces and linear operators see \cite[Chapter I, Sections 7 and 9]{defantfloret} and \cite[Chapter 21]{pietsch}.

It is not difficult to check that if $\frac{1}{p_1}+ \cdots + \frac{1}{p_n}   < \frac{1}{2}$, then the only $n$-linear operator of type $(p_1,\ldots,p_n)$ is $T=0$; and if $1 \leq  \frac{1}{p_1}+ \cdots + \frac{1}{p_n}$ then any continuous $n$-linear operator has type $(p_1,\ldots,p_n)$. In the same fashion, if $q < \frac{2}{n}$, then only the null operator has cotype $q$; and all continuous $n$-linear operators have cotype $q=\infty$. So we restrict ourselves to $n$-linear operators of type $(p_1, \ldots, p_n)$ with $\frac{1}{2} \leq \frac{1}{p_1}+ \cdots + \frac{1}{p_n}  < 1$ and to $n$-linear operators of cotype $q$ with $\frac{2}{n} \leq q < \infty$. Such an $n$-tuple $(p_1,\ldots,p_n)$  is said to be a {\it proper type for $n$-linear operators} and such a number $q$ is said to be a {\it proper cotype for $n$-linear operators}.

\begin{example}\label{ex1}\rm (a) We shall see soon (Theorem \ref{teoideal}) that multilinear operators of finite type have any proper type/cotype. As in the linear case, it is easy to check that finite rank $n$-linear operators have any proper type $(p_1, \ldots, p_n)$.\\
(b) Let $F$ be isomorphic to a Hilbert space $H$ and $\frac{1}{2} =  \frac{1}{p_1}+ \cdots + \frac{1}{p_n}$. A standard computation (orthogonality of the Rademacher functions, H\"older's inequality) shows that, regardless of the Banach spaces $E_1, \ldots, E_n$,  every operator $T \in \mathcal{L}(E_1,\ldots,E_n;F)$ has type $(p_1,\ldots,p_n)$ and $\|T\|_{\tau_{p_1,\ldots,p_n}} \leq \|T\| \cdot d(F,H)$, where $d$ is the Banach-Mazur distance.\\
(c) Let $n \in \mathbb{N}$ and $E_i$ be isomorphic to a Hilbert space $H_i$ for $i=1,\ldots,n$. As before, regardless of the Banach space $F$, any continuous $n$-linear operator $T$ from $E_1 \times \cdots \times E_n$ to $F$ has cotype $\frac{2}{n}$ and $\|T\|_{{\cal C}^n_{2/n}} \leq \|T\| \cdot \prod\limits_{i=1}^n d(E_i, H_i) $.
\end{example}

\begin{remark}\rm \label{firstex}
It is plain that if an $n$-linear operator $T$ has type $(p_1,\ldots,p_n)$, then $T$ has type $(r_1,\ldots,r_n)$ whenever $r_i \leq p_i$ for $i \in \{1,\ldots,n\}$ and $\frac{1}{r_1} + \cdots + \frac{1}{r_n} < 1$. Similarly, if $T$ has cotype $q$, then has cotype $r$ for $q \leq r$.
\end{remark}

The alert reader noticed that we did not mention the cotype of finite rank multilinear operators in Example \ref{ex1}. The reason is that this issue discloses a deep distinction between the linear and multilinear theories:

\begin{proposition}\label{postofinito} {\rm (a)} Finite rank multilinear operators have cotype 1.\\
{\rm (b)} Finite rank $n$-linear operators have any proper cotype if and only if $n = 1$ or $ n =2$.
\end{proposition}

\begin{proof} (a) From \cite[Proposition 3.1]{monat} it follows that multilinear forms (multilinear scalar-valued operators) have cotype 1. The case of arbitrary finite rank multilinear operators now follows easily.  \\
(b) The case $n = 1$ is obvious. Supposing $n = 2$ the result follows from (a) because $q = 1$ is the smallest proper cotype for bilinear operators. Conversely, for $n \geq 3$ consider the finite rank $n$-linear operator
$$T \colon \ell_n \times \stackrel{(n)}{\cdots} \times \ell_n \longrightarrow \mathbb{K}~,~T\left((\lambda_j^{(1)})_{j=1}^\infty, \ldots,  (\lambda_j^{(n)})_{j=1}^\infty \right) =  \sum_{j=1}^\infty \lambda_j^{(1)}\cdots \lambda_j^{(n)}. $$
By $(e_j)_{j=1}^\infty$ we denote the canonical unit vectors in $\ell_n$.
Assuming that $T$ has cotype $q$, there is a constant $C>0$ such that
\begin{align*}
k^{1/q} = \left(\sum_{j=1}^k |T(e_j,\ldots, e_j)|^q \right)^{1/q}\leq  C \cdot \left(\left\|(e_j)_{j=1}^k\right\|_{{\rm Rad}(\ell_n) } \right)^n= Ck
\end{align*}
for every $k \in \mathbb{N}$. This shows that $q \geq 1$, and combining with (a) we have
$$1 = \min\{q : T {\rm ~has ~cotype~} q\}. $$
In particular, $T$ has no proper cotype $\frac{2}{n} \leq q < 1$.
\end{proof}

So, unlike the linear case, in general finite rank multilinear operators {\it do not} have any proper cotype. Actually, more than a distinction between the linear and multilinear cases, we have a distinction between the linear/bilinear case, in which all finite rank operators have any proper cotype, and $n$-linear cases for $n > 2$.

The study of multi-ideals defined, or characterized, by the transformation of vector-valued sequences has been recently systematized in \cite{botelhocampos}. The classes we are studying here fit in this framework, so we apply the abstract results of \cite{botelhocampos} to derive important properties of multilinear operators of type $(p_1,\ldots,p_n)$ and of cotype $q$.  Keep in mind that, although stated for Banach ideals, the results of \cite{botelhocampos} hold for quasi-Banach ideals as well (cf. \cite[Remark 1.3(b)]{botelhocampos}).

From now on, $n$ is a positive integer not smaller than 2, $(p_1, \ldots, p_n)$ is a proper type and $q$ is a proper cotype for $n$-linear operators, that is, $\frac{1}{2} \leq \frac{1}{p_1}+ \cdots + \frac{1}{p_n} < 1$ and $q \geq \frac{2}{n}$.

\begin{theorem}\label{equivp}
 Given $(p_1,\ldots,p_n)$, the following statements are equivalent for a given $n$-linear operator $T\in\mathcal{L}(E_{1},\ldots,E_{n};F)$:\\
{\rm (i)} $T$ is an operator of type $(p_1,\ldots,p_n)$;\\
{\rm (ii)} $\left( T(x_j^{(1)},\ldots,x_j^{(n)}) \right)_{j=1}^\infty \in \mathrm{Rad}(F)$ whenever $(x_j^{(i)})_{j=1}^\infty \in \ell_{p_i}(E_i)$, $i \in \{1,\ldots,n\}$;\\
{\rm (iii)} The map
$$\widetilde{T}\left((x_j^{(1)})_{j=1}^\infty ,\ldots, (x_j^{(n)})_{j=1}^\infty \right) = \left(T(x_j^{(1)},\ldots,x_j^{(n)}) \right)_{j=1}^\infty$$
is a well-defined continuous $n$-linear operator from $\ell_{p_1}(E_1) \times \cdots \times \ell_{p_n}(E_n)$ to $\mathrm{Rad}(F)$.\\
\indent  In this case, $\Vert\widetilde{T}\Vert = \|T\|_{\tau_{p_1,\ldots,p_n}}$. Furthermore, the space $\mathrm{Rad}(F)$ in $\mathrm{(ii)}$ and $\mathrm{(iii)}$ can be equivalently replaced by $\mathrm{RAD}(F)$.
\end{theorem}

\begin{theorem}\label{equivq}
Given $q$, the following
statements are equivalent for a given $n$-linear operator $T\in\mathcal{L}(E_{1},\ldots,E_{n};F)$:\\
{\rm (i)} $T$ is an operator of cotype $q$;\\
{\rm (ii)} $\left( T(x_j^{(1)},\ldots,x_j^{(n)}) \right)_{j=1}^\infty \in \ell_q(F)$ whenever $(x_j^{(i)})_{j=1}^\infty \in \mathrm{Rad}(E_i)$, $i \in \{1,\ldots,n\}$;\\
 {\rm (iii)} The map
$$\widetilde{T}\left((x_j^{(1)})_{j=1}^\infty ,\ldots, (x_j^{(n)})_{j=1}^\infty \right) = \left(T(x_j^{(1)},\ldots,x_j^{(n)}) \right)_{j=1}^\infty$$
is a well-defined continuous $n$-linear operator from $\mathrm{Rad}(E_1) \times \cdots \times \mathrm{Rad}(E_n)$ to $\ell_q(F)$.\\
 \indent In this case, $\|\widetilde{T}\| = \|T\|_{{\cal C}^n_{q}}$. Furthermore, the spaces $\mathrm{Rad}(E_i)$ in $\mathrm{(ii)}$ and $\mathrm{(iii)}$ can be equivalently replaced by $\mathrm{RAD}(E_i)$, $i \in \{1,\ldots,n\}$.
\end{theorem}

Applying the language and the notation of \cite{botelhocampos}, it is noticeable that $\ell_{p_1}(\cdot),\ldots,\ell_{p_n}(\cdot)$, $\ell_{q}(\cdot)$ and $\mathrm{RAD}(\cdot)$ are finitely determined sequence classes. As $\mathrm{Rad}(\cdot) \subseteq \mathrm{RAD}(\cdot)$, the equivalences in both theorems above are particular cases of \cite[Proposition 1.4]{botelhocampos}.

\begin{theorem}\label{teoideal} $\tau_{p_1,\ldots,p_n}$ is a Banach ideal and ${\cal C}^n_q$ is a $q$-Banach ideal of $n$-linear operators.
\end{theorem}

\begin{proof} Applying the language and the notation of \cite{botelhocampos} once again, it is clear that $\ell_{p_1}(\cdot),\ldots,\ell_{p_n}(\cdot)$ and $\mathrm{Rad}(\cdot)$ are linearly stable sequence classes. By Theorem \ref{equivp}, an operator $T\in\mathcal{L}(E_{1},\ldots,E_{n};F)$ has type $(p_1,\ldots,p_n)$ if and only if it is $\left(\ell_{p_1}(\cdot),\ldots,\ell_{p_n}(\cdot);\mathrm{Rad}(\cdot)\right)$-summing; and, by Theorem \ref{equivq}, $T$ has cotype $q$ if and only if it is $\left(\mathrm{Rad}(\cdot),\ldots,\mathrm{Rad}(\cdot);\ell_{q}(\cdot)\right)$-summing. Notice that
\[\ell_{p_1}(\mathbb{K}) \cdots \ell_{p_n}(\mathbb{K}) = \ell_{p_1} \cdots \ell_{p_n} \stackrel{1}{\hookrightarrow} \ell_2 = \mathrm{Rad}(\mathbb{K})\]
because $\frac12 \leq \frac{1}{p_1} + \cdots +\frac{1}{p_n}$, and
\[\mathrm{Rad}(\mathbb{K}) \cdots  \mathrm{Rad}(\mathbb{K}) = \ell_{2} \cdots \ell_{2} \stackrel{1}{\hookrightarrow} \ell_q = \ell_q(\mathbb{K})\]
because $q \geq \frac{2}{n}$. From \cite[Theorem 2.6]{botelhocampos} it follows that $\tau_{p_1,\ldots,p_n}$ is a Banach $n$-ideal and ${\cal C}^n_q$ is a  $q$-Banach $n$-ideal.\end{proof}

The linear ideals of operators of some proper type/cotype are always Banach operator ideals, whereas in the multilinear case we sometimes have quasi-Banach ideals, for instance, for $n \geq 3$ and $\frac2n \leq q < 1$, ${\cal C}_q^n$ is a $q$-Banach $n$-ideal.

We have presented examples of multilinear operators having any proper type/cotype (operators in Example \ref{ex1}(b) and (c), and finite type operators in Theorem \ref{teoideal}). Next we give non-trivial examples of operators having no proper type/cotype:

\begin{example}\rm \label{ex1}
(a) Consider the continuous bilinear operator
$$T \colon \ell_1 \times \ell_1 \longrightarrow \ell_1~, ~T\left( (\lambda_j)_{j=1}^\infty, (\eta_j)_{j=1}^\infty \right) = (\lambda_j\eta_j)_{j=1}^\infty.$$
Assuming that $T$ has some proper type $(p_1,p_2)$, there is a constant $C>0$ such that
\begin{align*}
k = \left\|\left(T(e_j,e_j) \right)_{j=1}^k \right\|_{{\rm Rad}(\ell_1)} 
\leq C \cdot \left( \sum_{j=1}^k \|e_j\|_{\ell_1}^{p_1}\right) ^{1/{p_1}} \cdot \left( \sum_{j=1}^k \|e_j\|_{\ell_1}^{p_2}\right) ^{1/{p_2}} = C \cdot k^{\left(1/{p_1} +  1/{p_2}\right)}
\end{align*}
for every $k \in \mathbb{N}$. This contradiction -- remember that $\frac{1}{p_1} + \frac{1}{p_2} < 1$ --  shows that $T$ has no proper type.

\noindent (b) Consider the continuous bilinear operator $T\colon \ c_0 \times c_0 \longrightarrow c_0$ defined in the same fashion as in (a). Using again that $T(e_j, e_j) = e_j$ for every $j$, assuming that $T$ has some proper cotype $q$, $1 \leq q < \infty$, there exists a constant $M>0$ such that $k^{1/q} \leq M$ for every $k \in \mathbb{N}$. This contradiction shows that $T$ has no proper cotype.
\end{example}

It is time to give examples of multilinear operators having some proper type/cotype but not having any proper type/cotype. To do so, consider the $(n+1)$-linear operator $\varphi_1 \otimes \cdots \otimes \varphi_{n} \otimes u \in \mathcal{L}(E_{1},\ldots,E_{n+1};F)$ given by
 \begin{equation}\varphi_1 \otimes \cdots \otimes \varphi_{n} \otimes u (x_1,\ldots,x_{n+1}) = \varphi_1(x_1)\cdots\varphi_{n}(x_{n})u(x_{n+1}),\label{oper}
 \end{equation}
where $0 \neq\varphi_i \in E_i^{'}$, $i=1,\ldots,n$, and $0 \neq u \in \mathcal{L}(E_{n+1};F)$. It is somewhat expected that $\varphi_1 \otimes \cdots \otimes \varphi_{n} \otimes u$ has a proper type $(p_1,\ldots,p_{n},p)$ whenever $u$ has type $1< p \leq 2$. Quite surprisingly, we shall see below that, actually, $\varphi_1 \otimes \cdots \otimes \varphi_{n} \otimes u$ has much more proper types than these ones. This result expresses a typically multilinear feature of the theory.

Given a linear operator $u \in {\cal L}(E;F)$, we define the usual parameter
$$p_u := \sup\{p : u {\rm ~has~type~} p\} \in [1,2].$$
It is well known that this supremum is not necessarily attained (see, e.g., \cite[p.\,305]{djt}).

\begin{proposition}\label{optensort} Let $\varphi_1 \otimes \cdots \otimes \varphi_{n} \otimes u$ be the  $(n+1)$-linear operator defined in {\rm (\ref{oper})}.\\
{\rm (a)} $\varphi_1 \otimes \cdots \otimes \varphi_{n} \otimes u$ has any proper type $(p_1,\ldots,p_{n+1})$ whenever $p_k = p_u$ for some $k \in \{1,\ldots,n+1\}$.\\
{\rm (b)} If $u$ has cotype $q \geq 2$, then $\varphi_1 \otimes \cdots \otimes \varphi_{n} \otimes u$ has cotype $\frac{2q}{2+qn}$.
\end{proposition}

\begin{proof}(a) Let $(x_j^{(i)})_{j=1}^\infty \in \ell_{p_i}(E_i)$ for $i=1,\ldots,n+1$, be given. Defining $p$ by \begin{equation}\label{eqq}
\frac{1}{p} = \frac{1}{p_1}+ \cdots + \frac{1}{p_{n+1}},
\end{equation}
since $\left(\varphi_i(x_j^{(i)})\right)_{j=1}^\infty \in \ell_{p_i}$ for $i=1,\ldots,n$, from H\"older's inequality we obtain
\[\left(\varphi_1(x_j^{(1)}) \cdots \varphi_n(x_j^{(n)})x_j^{(n+1)}\right)_{j=1}^\infty \in \ell_{p}(E_{n+1}).\]
As $p_k = p_u$ for some $k \in \{1,\ldots,n+1\}$, from (\ref{eqq}) it follows that $p < p_u$. So $u$ has type $p$, therefore
\begin{align*}
 \left(\varphi_1 \otimes \cdots  \otimes \varphi_{n} \otimes u  (x_j^{(1)},\ldots,x_j^{(n+1)})\right)_{j=1}^\infty & =\left(\varphi_1(x_j^{(1)}) \cdots \varphi_n(x_j^{(n)})u(x_j^{(n+1)}
)\right)_{j=1}^\infty \\
&=\left(u\left(\varphi_1(x_j^{(1)}) \cdots \varphi_n(x_j^{(n)})x_j^{(n+1)}
\right)\right)_{j=1}^\infty \in \mathrm{Rad}(F).
\end{align*}
(b) Set $\overline{q} := \frac{2q}{2+qn}$ and let $k \in \mathbb{N},~ x_j^{(i)} \in E_i$, $i = 1, \ldots, n, j = 1, \ldots, k$ be given.
As $$\frac{1}{\overline{q}} = \frac{2+qn}{2q} = \frac{n}{2} + \frac{1}{q},$$ we obtain again from H\"older's inequality,
\begin{align*}
 &\left(\sum_{j=1}^k \left\| \varphi_1(x_j^{(1)})\cdots\varphi_{n}(x_j^{(n)})u(x_j^{(n+1)}
)\right\|^{\overline{q}}\right)^{1/{\overline{q}}}\\
& \leq  \left( \sum_{j=1}^k|\varphi_1(x_j^{(1)})|^2\right)^{1/2} \cdots \left( \sum_{j=1}^k|\varphi_{n}(x_j^{(n)})|^2\right)^{1/2}\cdot \left( \sum_{j=1}^k\|u(x_j^{(n+1)}) \|^q \right)^{1/q}\\
& \leq \|u\|_{C_{q}}\cdot \|\varphi_1\| \cdots \|\varphi_{n}\|\cdot \prod_{i=1}^{n+1} \left(\int_0^1 \left\Vert \sum_{j=1}^k r_j(t)x_j^{(i)} \right\Vert^2 dt\right)^{1/2}.
\end{align*}
\end{proof}

Our next aim is to provide multilinear operators having some proper types/cotypes but not having any proper type/cotype, and to show that the cotype computed in Proposition \ref{optensort}(b) is optimal in general.
Some notation and definitions are needed. We say that a bounded sequence $\left(y_j\right)_{j=1}^\infty$ in a Banach space $E$ is \emph{seminormalized} if $\inf\limits_j \|y_j\| >0$.
\begin{definition}\rm
Let $1 \leq p,q < \infty$. We say that a Banach space $E$ has the \emph{sign lower $p$-estimate property} (respectively the \emph{sign upper $q$-estimate property}) if there exist a seminormalized sequence $\left(y_j\right)_{j=1}^\infty$ in $E$ and a constant $R>0$ such that, for all $k \in \mathbb{N}$,
\begin{equation}\label{rep}
R k^{1/p} \leq \left\|\sum_{j=1}^k \varepsilon_jy_j\right\|_E, \ \ \left(\mathrm{respectively\ \ }\left\|\sum_{j=1}^k \varepsilon_jy_j\right\|_E \leq R k^{1/q}\right)\ \ \mathrm{for\ all\ }\varepsilon_j = \pm 1.
\end{equation}
\end{definition}

Let us see that the properties above are easy to be found:

\begin{example}\rm (a) We take advantage of the properties $S_q$ and $T_p$ introduced by Knaust and Odell in \cite{knaust} and developed by many authors, see, e.g., \cite{castillo, dimant, gonzalo}. Let $E$ be a reflexive Banach space. If $E$ has  property $S_q$ then $E$ has the sign upper $q$-estimate property, and if $E$ has property $T_p$ then $E$ has the sign lower $p$-estimate property. For instance, $\ell_r$ spaces, $1 < r < \infty$, have the sign upper (and lower) $r$-estimate property.\\
(b) For further concrete examples, let $\mu$ be a measure for which there exists a sequence of pairwise disjoint mensurable sets $(A_n)_{n=1}^\infty$ with $0 < \mu(A_n) < \infty$ for all $n$ and $\sum_{n=1}^\infty \mu(A_n) < \infty$. Given a Banach space $X$, choose   $x \in X$ with $\|x\| = 1$. Working with the sequence $f_n = \frac{1}{\mu(A_n)^{1/r}}\cdot \mathcal{X}_{A_n}\cdot x$, $n \in \mathbb{N}$, we see that the Lebesgue-Bochner space $L_r(\mu, X)$ has the sign upper/lower $r$-estimate properties for every $1 \leq r < \infty$. In particular, the Lebesgue spaces $L_r[a,b]$ have the sign upper/lower $r$-estimate properties.
\end{example}

Now we are ready to achieve the announced task:

\begin{proposition}\label{optimal}
Let $\varphi_1 \in E_1', \ldots, \varphi_{n} \in E_n'$ and $E_{n+1}$ be a Banach space.\\
{\rm (a)} If $E_{n+1}$ has type $1< p \leq 2$ and the sign lower $p$-estimate property, then $\varphi_1 \otimes \cdots \otimes \varphi_{n} \otimes id_{E_{n+1}}$ has any proper type $(p_1,\ldots,p_{n+1})$ whenever $p_k = p$ for some $k \in \{1,\ldots,n+1\}$, and has no proper type $(r_1,\ldots,r_{n+1})$ with $\frac{1}{2} \leq \frac{1}{r_1}+ \cdots + \frac{1}{r_{n+1}} < \frac{1}{p}$.\\
{\rm (b)} If $E_{n+1}$ has cotype $q \geq2$ and the sign upper $q$-estimate property, then
$$\frac{2q}{2+qn} = \min \left\{r : \varphi_1 \otimes \cdots \otimes \varphi_{n} \otimes id_{E_{n+1}}\ \mathrm{has\ cotype}\ r\right\}.$$
\end{proposition}

\begin{proof} (a) According to Proposition \ref{optensort}, the operator has such proper types $(p_1,\ldots,p_{n+1})$. For simplicity, we shall prove what is left in the bilinear case ($n=1$). Suppose that $E_{2}$ has type $p \leq 2$ and the sign lower $p$-estimate. Let $r_1,r_2$ be such that $\frac{1}{2} \leq \frac{1}{r_1}+ \frac{1}{r_2} < \frac{1}{p}$. Let $(y_j)_{j=1}^\infty$ be a seminormalized sequence in $E_2$ and $R>0$ be as in the definition of the sign lower $p$-estimate property. Say $\|y_j\| \leq K$ for every $j$. Choose $a \in E_1$ such that $\varphi_1(a) \neq 0$. Assuming that $\varphi_1 \otimes id_{E_{2}}\colon E_1 \times E_2 \longrightarrow E_2$ has type $(r_1,r_2)$, there exists a constant $C>0$ such that
\begin{align*} \left|\varphi_1(a)\right|Rk^{1/p}& =\left|\varphi_1(a)\right| \left(\frac{1}{2^n}\left[2^n \left(Rk^{1/p} \right)^2\right]\right)^{1/2} \leq \left|\varphi_1(a)\right| \left(\frac{1}{2^n}\cdot\sum_{\varepsilon_j = \pm 1}\left\|\sum_{j=1}^k \varepsilon_j y_j \right\|^2 \right)^{1/2}\\&= \left|\varphi_1(a)\right| \left(\int_0^1 \left\Vert \sum_{j=1}^k r_j(t)y_j \right\Vert^2 dt\right)^{1/2} = \left(\int_0^1 \left\Vert \sum_{j=1}^k r_j(t)\varphi_1(a)y_j \right\Vert^2 dt\right)^{1/2}\\& = \left(\int_0^1 \left\Vert \sum_{j=1}^k r_j(t)\varphi_1 \otimes id_{E_2}(a,y_j) \right\Vert^2 dt\right)^{1/2}\\& \leq C \left( \sum_{j=1}^k \|a\|^{r_1}\right) ^{1/{r_1}}\cdot  \left( \sum_{j=1}^k \|y_j\|^{r_2}\right) ^{1/{r_2}} \leq CK \cdot \|a\|\cdot k^{1/{r_1}+1/{r_2}}
\end{align*}
for every $k \in \mathbb{N}$ -- a contradiction because $ \frac{1}{r_1}+ \frac{1}{r_2} < \frac{1}{p}$.\\
\noindent (b) According to Proposition \ref{optensort}, the operator has cotype $\frac{2q}{2+qn}$. Observe that there is nothing else to do if $q = 2$. Indeed, in this case $\frac{2q}{2+qn} = \frac{2}{n+1}$ which is the best possible cotype $\varphi_1 \otimes \cdots \otimes \varphi_{n} \otimes id_{E_{n+1}}$ can have, a cotype we already know it has. Assume that $E_{n+1}$ has cotype $q > 2$. Let $\overline{q} = \frac{2q}{2+qn}$ and $1 \leq r < \overline{q}$. Choose $a_j \in E_j$ such that $\varphi_j(a_j)\neq 0$, $j = 1, \ldots, n$. Let $(y_j)_{j=1}^\infty$ be a seminormalized sequence in $E_{n+1}$ and $R>0$ be as in the definition of the sign upper $p$-estimate property. Assuming that $\varphi_1 \otimes \cdots \otimes \varphi_{n} \otimes id_{E_{n+1}}$ has cotype $r$ with constant $C$, a computation similar to the one performed in the proof (a) yields that
$$|\varphi_1(a_1)| \cdots |\varphi_n(a_n)| \cdot \inf_j \|y_j\|\cdot k^{1/r} \leq CR\cdot \|a_1\|\cdots \|a_n\|\cdot k^{n/2+ 1/q}$$
for every $k \in \mathbb{N}$. As $\frac{n}{2} + \frac{1}{q} = \frac{1}{\overline{q}} < \frac{1}{r}$, this is absurd.
\end{proof}

In part (a) of the proposition above, for $1 < p < 2$ it is always possible to choose $r_1,\ldots,r_{n+1}$ with $\frac{1}{2} \leq \frac{1}{r_1}+ \cdots + \frac{1}{r_{n+1}} < \frac{1}{p}$, and then $\varphi_1 \otimes \cdots \otimes \varphi_{n} \otimes id_{E_{n+1}}$ fails to have some proper type $(r_1, \ldots, r_{n+1})$.

\bigskip

\noindent {\bf Open question 1.} Under the assumptions of Proposition \ref{optimal}(a), does $\varphi_1 \otimes \cdots \otimes \varphi_{n} \otimes id_{E_{n+1}}$ have some/any proper type $(r_1,\ldots,r_{n+1})$  with $\frac{1}{p} \leq \frac{1}{r_1} + \cdots + \frac{1}{r_{n+1}}$?

\section{Inclusion relationships with related multi-ideals}

\hspace{0,6cm}  In this section we establish the relationships of the multi-ideals $\tau_{p_1,\ldots,p_n}$ and ${\cal C}^n_q$ with the multi-ideals generated by the standard methods of generalizing a given linear operator ideal to the multilinear setting (see, e.g. \cite{note,pietsch83}), starting with the ideals $\tau_p$ of linear operators of type $p$ and ${\cal C}_q$ of linear operators of cotype $q$. Some definitions and notations are needed now.

For given operator ideals $\mathcal{I},\mathcal{I}_1,\ldots, \mathcal{I}_n$, and a mapping $T \in \mathcal{L}\left(E_{1},\ldots,E_{n};F\right)$, the following constructions are standard:\\
$\bullet$ (Factorization method) $T$ is called to be of type $\mathcal{L}(\mathcal{I}_1,\ldots, \mathcal{I}_n)$, and in this case we write $T \in \mathcal{L}(\mathcal{I}_1,\ldots, \mathcal{I}_n)\left(E_{1},\ldots,E_{n};F\right)$, if there are Banach spaces $G_1,\ldots,G_n$, a continuous $n$-linear operator $B \in \mathcal{L}\left(E_{1},\ldots,E_{n};F\right)$ and linear operators $u_i \in \mathcal{I}_i\left(E_{i};G_i\right)$, $i = 1,\ldots,n$, such that $T=B(u_1,\ldots,u_n)$. If $\mathcal{I}_1,\ldots, \mathcal{I}_n$ are normed operator ideals, we define
\[ \|T\|_{\mathcal{L}(\mathcal{I}_1,\ldots, \mathcal{I}_n)}= \inf \|B\|\cdot\|u_1\|_{\mathcal{I}_1}\cdots \|u_n\|_{\mathcal{I}_n},\]
taking the infimum over all possible factorizations of $T$ in this fashion above.\\
$\bullet$ (Linearization method) $T$ is said to be of type $[\mathcal{I}_1,\ldots, \mathcal{I}_n]$, and in this case we write $T \in [\mathcal{I}_1,\ldots, \mathcal{I}_n]\left(E_{1},\ldots,E_{n};F\right)$, if $I_i(T)$ belongs to $\mathcal{I}_i(E_i;\mathcal{L}(E_{1},\overset{[i]}{\ldots},E_{n};F))$ for $i= 1,\ldots,n$, where the operator $I_i \colon \mathcal{L}\left(E_{1},\ldots,E_{n};F\right) \longrightarrow \mathcal{L}(E_i;\mathcal{L}(E_{1},\overset{[i]}{\ldots},E_{n};F))$ is given by $$I_i(T)(x_i)(x_1,\overset{[i]}{\ldots},x_n)=T(x_1,\ldots,x_n),$$
and $\overset{[i]}{\ldots}$ means that the $i$-th coordinate is not involved. If $\mathcal{I}_1,\ldots, \mathcal{I}_n$ are normed ideals, define
$$ \|T\|_{[\mathcal{I}_1,\ldots, \mathcal{I}_n]}= \max \left\{\|I_1(T)\|_{\mathcal{I}_1},\ldots ,\|I_n(T)\|_{\mathcal{I}_n}\right\}.$$
$\bullet$ (Composition ideals) $T$ belongs to $\mathcal{I} \circ \mathcal{L}$, denoted $T \in \mathcal{I} \circ \mathcal{L}\left(E_{1},\ldots,E_{n};F\right)$, if there are a Banach space $G$, an operator $u \in \mathcal{I}(G;F)$ and an $n$-linear mapping $B \in \mathcal{L}\left(E_{1},\ldots,E_{n};G\right)$ such that $T=u \circ B$. If $\mathcal{I}$ is a normed ideal, define
\[\|T\|_{\mathcal{I} \circ \mathcal{L}}=\inf \left\{\|B\|\cdot\|u\|_\mathcal{I} : T=u \circ B,\, u \in \mathcal{I}(G;F) \mathrm{\ and\ } B \in \mathcal{L}\left(E_{1},\ldots,E_{n};G\right) \right\}.\]

It is well-known that if $\mathcal{I}$ and $\mathcal{I}_1,\ldots, \mathcal{I}_n$ are normed (Banach) operator ideals, then $\left(\mathcal{L}(\mathcal{I}_1,\ldots, \mathcal{I}_n), \|\cdot\|_{\mathcal{L}(\mathcal{I}_1,\ldots, \mathcal{I}_n)}\right)$ is a quasi-normed (quasi-Banach) multi-ideal and both \\$\left([\mathcal{I}_1,\ldots, \mathcal{I}_n], \|\cdot\|_{[\mathcal{I}_1,\ldots, \mathcal{I}_n]}\right)$ and $\left(\mathcal{I} \circ \mathcal{L}, \|\cdot\|_{\mathcal{I} \circ \mathcal{L}}\right)$ are normed (Banach) multi-ideals. It is also known that $\mathcal{L}(\mathcal{I}_1,\ldots, \mathcal{I}_n)$ $\subseteq [\mathcal{I}_1,\ldots, \mathcal{I}_n]$.

\begin{theorem}\label{inclus} The following inclusion relations hold for the multi-ideals of operators of proper type $(p_1,\ldots,p_n)$ and proper cotype $q$:\\
{\rm (i)} If  $q \geq 2$, then $\mathcal{C}_q \circ \mathcal{L} \subseteq {\cal C}_q^n$ and $\|T\|_{{\cal C}_q^n} \leq \|T\|_{\mathcal{C}_q \circ \mathcal{L}}$.\\
{\rm (ii)} If $\frac{1}{p} \leq \frac{1}{p_1}+ \cdots + \frac{1}{p_n}$, then $\tau_p \circ \mathcal{L} \subseteq \tau_{p_1,\ldots,p_n}$ and $\|T\|_{\tau_{p_1,\ldots,p_n}} \leq \|T\|_{\tau_p \circ \mathcal{L}}$.\\
{\rm (iii)} If $\frac{1}{q} \leq \frac{1}{q_1}+ \cdots + \frac{1}{q_n}$, then  $\mathcal{L}\left({\cal C}_{q_1},\ldots,{\cal C}_{q_n}\right) \subseteq {\cal C}_q^n$ and $\|T\|_{{\cal C}_q^n} \leq \|T\|_{\mathcal{L}\left({\cal C}_{q_1},\ldots,{\cal C}_{q_n}\right)}$.
\end{theorem}

\begin{proof}(i) Given $T \in \mathcal{C}_q \circ \mathcal{L}(E_1, \ldots, E_n;F)$, write $T = u \circ B$ with $u \in \mathcal{C}_q(G;F)$ and $B \in \mathcal{L}\left(E_{1},\ldots,E_{n};F\right)$. By \cite[Theorem 3.5]{botelhocampos} and Theorem \ref{equivq}, respectively, the maps $\widetilde{B}$ and $\widetilde{u}$, respectively, are well-defined according to the diagram below.
\begin{center}
\begin{minipage}{6cm}
\xymatrix@C0pt@R30pt{
\mathrm{Rad}(E_1) \times \cdots \ar@/_/[d]_*{\widetilde{B}} \times  \mathrm{Rad}(E_n)  \ar[rrrrrrrrrr]^-*{\widetilde{T}}& & & & & & & & & & \ell_q(F) \\
 \mathrm{Rad}(G) \ar@/_/[urrrrrrrrrr]_*{\widetilde{u}}   & & & & & & &
}
\end{minipage}
\end{center}
Thus $\widetilde{T}$ is well defined by $\widetilde{u} \circ \widetilde{B}$ and $T$ has cotype $q$ by Theorem \ref{equivq}. Furthermore,
\[\|T\|_{{\cal C}_q^n} = \|\widetilde{T}\| \leq \|\widetilde{B}\|\cdot\|\widetilde{u}\| = \|B\|\cdot\|u\|,\]
for any factorization of $T$ in this fashion. Taking the infimum over all such factorizations we get $\|T\|_{{\cal C}_q^n} \leq \|T\|_{\mathcal{C}_q \circ \mathcal{L}}.$

\noindent (ii) and (iii) follow from straightforward computations using the definitions and H\"older's inequality.
\end{proof}

Observe that the inclusion $\mathcal{L}\left(\tau_{p_1},\ldots,\tau_{p_n}\right) \subseteq \tau_{p_1,\ldots,p_n}$ is trivial because it only makes sense for $1 \leq p_i \leq 2$, $i=1,\ldots,n$, and in this case we have $1 \leq \frac{n}{2} \leq \frac{1}{p_1} + \cdots \frac{1}{p_n},$ a situation in which $\tau_{p_1,\ldots,p_n} = \mathcal{L}$. We are thus impelled to investigate eventual inclusion relations between $\tau_{p_1,\ldots,p_n}$ and $\mathcal{L}\left(\tau_{r_1},\ldots,\tau_{r_n}\right)/
\left[\tau_{r_1},\ldots,\tau_{r_n}\right]$ with $p_j \neq r_j$ for some $j$. To avoid trivialities we shall henceforth suppose that $r_j > 1$ for some $j$. We are about to see that nontrivial relations of this type, if any, are rare. Moreover, we shall see that the inclusion in Theorem \ref{inclus}(iii) can be strict.

\begin{proposition} {\rm (a)} $\tau_{p_1, \ldots,p_n} \not\subseteq [\tau_{r_1},\ldots, \tau_{r_n}]$, hence $\tau_{p_1, \ldots,p_n} \not\subseteq {\cal L}(\tau_{r_1},\ldots, \tau_{r_n})$, for every proper type $(p_1, \ldots, p_n)$ for $n$-linear operators and all $1\leq r_1,\ldots, r_n \leq 2$.\\
{\rm (b)} ${\cal L}(\tau_{r_1},\ldots, \tau_{r_n}) \not\subseteq \tau_{p_1, \ldots,p_n}$ whenever  $(p_1, \ldots, p_n)$  is proper type for $n$-linear operators, $1\leq r_1,\ldots, r_n \leq 2$ and $\frac{1}{p_1} + \cdots +\frac{1}{p_n} < \max\left\{\frac{1}{r_1}, \ldots, \frac{1}{r_n}\right\}$.\\
{\rm (c)} $[{\cal C}_{q_1},\ldots,{\cal C}_{q_n}] \not\subseteq {\cal C}_q^n$ for every $ q \geq 1$, therefore $\mathcal{L}\left({\cal C}_{q_1},\ldots,{\cal C}_{q_n}\right) \varsubsetneq {\cal C}_q^n$ whenever  $\frac{1}{q} \leq \min\left\{1, \frac{1}{q_1}+ \cdots + \frac{1}{q_n}\right\}$.
\end{proposition}

\begin{proof} (a) Assume, wlog, that $r_1 > 1$. Let $E$ be a Banach space not having type $r_1$ and consider the continuous $n$-linear operator
$$T\colon E \times E' \times \mathbb{K}^{n-2} \longrightarrow \mathbb{K}~,~T(x, \varphi, \lambda_1,\ldots, \lambda_{n-2}) = \lambda_1\cdots\lambda_{n-2} \varphi(x).$$ As $T$ has finite rank, it has any proper type $(p_1,\ldots, p_n)$ (Example \ref{ex1}(a)). Suppose that $T$ belongs to the multi-ideal $[\tau_{r_1},\ldots, \tau_{r_n}]$. In this case the linear operator $I_1(T) \colon E \longrightarrow {\cal L}(E', \mathbb{K}^{n-2}; \mathbb{K})$ has type $r_1$. Since $I_1(T)$ is linear and, by the Hahn-Banach Theorem we have $\|I_1(T)(x)\| = \|x\|$ for every $ x \in E, $
this implies that $E$ has type $r_1$ -- a contradiction.\\
(b) Assume, wlog, that  $\frac{1}{p_1} + \cdots +\frac{1}{p_n} < \frac{1}{r_n}$. Let $E$ be a Banach space of type $r_n$ and with the Rademacher lower-$r_n$-estimate property. Choosing $0 \neq \varphi \in E'$, by Proposition \ref{optimal}(a) we know that $\varphi\otimes \stackrel{(n-1)}{\cdots}\otimes \varphi \otimes id_E$ fails to have type $(p_1, \ldots, p_n)$. On the other hand, considering the $n$-linear operator
$$B \colon \mathbb{K} \times \stackrel{(n-1)}{\cdots}\times \mathbb{K} \times E \longrightarrow E~,~B(\lambda_1, \ldots, \lambda_{n-1}, x) =  \lambda_1 \cdots \lambda_{n-1}x,$$
the factorization
$$\varphi\otimes \stackrel{(n-1)}{\cdots}\otimes \varphi \otimes id_E = B \circ(\varphi, \stackrel{(n-1)}{\ldots}, \varphi, id_E) $$
shows that $\varphi\otimes \stackrel{(n-1)}{\cdots}\otimes \varphi \otimes id_E$ belongs to ${\cal L}(\tau_{r_1}, \ldots, \tau_{r_1})$.\\
(c) Let $E$ be a Banach space not having cotype $q_1$, $0 \neq \varphi \in E'$ and consider the continuous $n$-linear operator $T$ defined in the proof of (a). As $q \geq 1$, $T$ has cotype $q$ because it is a finite rank operator (Example \ref{ex1}(a)). And $T$ does not belong to $[{\cal C}_{q_1},\ldots,{\cal C}_{q_n}]$ because otherwise, proceeding as in the proof of (a),  $E$ would have cotype $q_1$. The second assertion follows from the first and Proposition \ref{inclus}(iii).
\end{proof}

\medskip

\noindent{\bf Open question 2.} Is it true that ${\cal L}(\tau_{r_1},\ldots, \tau_{r_n}) \not\subseteq \tau_{p_1, \ldots,p_n}$ for every proper type $(p_1, \ldots, p_n)$ for $n$-linear operators and all $1\leq r_1,\ldots, r_n \leq 2$ with $r_j >1$ for some $j$?

\medskip

Just for the record we state the following coincidences:

\begin{corollary}
If $E_i$ has type $p_i$, $i \in \{1,\ldots,n\}$, then, regardless of the Banach space $F$,
\begin{align*}\tau_{p_1,\ldots,p_n}\left(E_{1},\ldots,E_{n};F\right) &= {\cal L}(\tau_{p_1},\ldots, \tau_{p_n})\left(E_{1},\ldots,E_{n};F\right) \\&=  [\tau_{p_1},\ldots, \tau_{p_n}]\left(E_{1},\ldots,E_{n};F\right)= \mathcal{L}\left(E_{1},\ldots,E_{n};F\right).
\end{align*}
If $E_i$ has cotype $q_i$, $i \in \{1,\ldots,n\}$, and $\frac{1}{q} \leq \frac{1}{q_1}+ \cdots + \frac{1}{q_n}$, then,  regardless of the Banach space $F$,
\begin{align*}{\cal C}_q^n\left(E_{1},\ldots,E_{n};F\right) &= {\cal L}({\cal C}_{q_1},\ldots, {\cal C}_{q_n})\left(E_{1},\ldots,E_{n};F\right) \\&=  [{\cal C}_{q_1},\ldots, {\cal C}_{q_n}]\left(E_{1},\ldots,E_{n};F\right)= \mathcal{L}\left(E_{1},\ldots,E_{n};F\right).
\end{align*}
\end{corollary}

\begin{proof} Just combine the factorization $T=T\circ(id_{E_1},\ldots,id_{E_n})$ for any $T \in \mathcal{L}\left(E_{1},\ldots,E_{n};F\right)$ with our previous results/remarks (for (b) use Proposition \ref{inclus}(iii)).
\end{proof}

\section{Maximality and Aron-Berner stability}

In this last section we prove that the multi-ideals $\tau_{p_1,\ldots,p_n}$ and ${\cal C}_q^n$ enjoy and important typical multilinear property, namely, the Aron-Berner stability. We need some notation/terminology.

Let $\mathcal{U}$ be an ultrafilter on a set $I$. The ultraproduct (along $\mathcal{U}$) of a family $(E_\lambda)_{\lambda \in I}$ of Banach spaces $E_\lambda$, denoted by $(E_\lambda)_{\mathcal{U}}$, is a Banach space with its usual norm
$$\|(x_\lambda)_{\mathcal{U}}\|_{(E_\lambda)_{\mathcal{U}}} = \lim_{\mathcal{U}} \|x_\lambda\| {\rm ~ for~every~} (x_\lambda)_{\mathcal{U}} \in (E_\lambda)_{\mathcal{U}}.$$
For background on ultraproducts of Banach spaces we refer to  \cite[Chapter 8]{djt}. Given a family $(T_\lambda)_{\lambda \in I}$ of continuous $n$-linear operators $T_\lambda\colon E_{1,\lambda} \times \cdots \times E_{n,\lambda}\longrightarrow F_\lambda$ such that $\sup_{\lambda \in I}\|T_\lambda\| < \infty$, we can define the $n$-linear operator $\overline{(T_\lambda)}_{\mathcal{U}}\colon (E_{1,\lambda})_{\mathcal{U}} \times \cdots \times (E_{n,\lambda})_{\mathcal{U}} \longrightarrow (F_{\lambda})_{\mathcal{U}}$ by
\[\overline{(T_\lambda)}_{\mathcal{U}} \left(\left(x_\lambda^1\right)_{\mathcal{U}},\ldots,\left(x_\lambda^n\right)_{\mathcal{U}}\right) := \left(T_\lambda\left(x_\lambda^1,\ldots,x_\lambda^n\right)\right)_{\mathcal{U}}.\]
A straightforward computation gives $\left\|\overline{(T_\lambda)}_{\mathcal{U}}\right\| = \lim_{\mathcal{U}}\|T_\lambda\|$.

A multi-ideal $\mathcal{M}$ is called:\\
$\bullet$ {\it regular} if $J_F \circ T \in \mathcal{M}(E_1,\ldots,E_n;F'')$ implies that $T \in \mathcal{M}(E_1,\ldots,E_n;F)$ and $\|T\|_{\mathcal{M}} = \|J_F \circ T\|_{\mathcal{M}}$. \\
$\bullet$ {\it ultrastable} if for any family  $T_\lambda \in \mathcal{M}(E_{1,\lambda},\ldots,E_{n,\lambda};F_\lambda)$ such that $\|T_\lambda\|_{\mathcal{M}} \leq C$ for every $\lambda \in I$, it holds $\overline{(T_\lambda)}_{\mathcal{U}} \in \mathcal{M}\left((E_{1,\lambda})_{\mathcal{U}},\ldots,(E_{n,\lambda})_{\mathcal{U}};(F_{\lambda})_
{\mathcal{U}}\right)$ and $\left\|\overline{(T_\lambda)}_{\mathcal{U}}\right\|_\mathcal{M} \leq \sup_\lambda\|T_\lambda\|_\mathcal{M}$.\\
$\bullet$ {\it maximal} if
\[\|T\|_{\mathcal{M}^{\mathrm{max}}} := \sup\left\{\left\|Q_L^F \circ T|_{M_1\times \cdots \times M_n}\right\|_\mathcal{M}:  M_j \in \mathrm{FIN}(E_j), L \in \mathrm{COFIN}(F) \right\}< \infty\]
implies that $T \in \mathcal{M}(E_1, \ldots, E_n;F)$ and $\|T\|_{\mathcal{M}} = \|T\|_{\mathcal{M}^{\mathrm{max}}}$, where $Q_L^F\colon F \longrightarrow F/L$ is the canonical quotient map, $\mathrm{FIN}(E_j)$ is the set of all finite-dimensional subspaces of $E_j$ and $\mathrm{COFIN}(F)$ is the set of all finite-codimensional closed subspaces of $F$.
Is well-known that the linear operator ideals $\tau_p$ and ${\cal C}_q$ are maximal (see, for instance, \cite[p.\,203]{defantfloret}). In \cite[Section 4]{floret02} it was proved that a multi-ideal $\mathcal{M}$ is maximal if and only if it is regular and ultrastable. We will use this characterization to prove the maximality of the multi-ideals of proper type/cotype:

\begin{theorem}\label{maxi}
The multi-ideals $\tau_{p_1,\ldots,p_n}$ and ${\cal C}_q^n$ are regular and ultrastable and, therefore, are maximal.
\end{theorem}

\begin{proof}
It is immediate that $\tau_{p_1,\ldots,p_n}$ and ${\cal C}_q^n$ are regular. Let us prove the ultrastability of ${\cal C}_q^n$ (the case of $\tau_{p_1,\ldots,p_n}$ is analogous). For any $\lambda \in I$, let  $T_\lambda \in {\cal C}_q^n(E_{1,\lambda},\ldots,E_{n,\lambda};F_\lambda)$ be such that $\sup_\lambda\|T_\lambda\|_{{\cal C}_q^n} \leq C$. Then for any $k \in \mathbb{N}$ and all $\left(x_\lambda^{i,j}\right)_{\mathcal{U}} \in (E_{i,\lambda})_\mathcal{U} $, $j= 1,\ldots,k$, $i = 1, \ldots,n$,
\begin{align*}
& \left(\sum_{j=1}^k \left\Vert \overline{(T_\lambda)}_{\mathcal{U}} \left(\left(x_\lambda^{1,j}\right)_{\mathcal{U}},\ldots,\left(x_\lambda^{n,j}\right)_{\mathcal{U}}\right) \right\Vert^q_{(F_\lambda)_\mathcal{U}}\right)^{1/q}\\
& ~~~~~= \left(\sum_{j=1}^k \left\Vert \left(T_\lambda\left(x_\lambda^{1,j},\ldots,x_\lambda^{n,j}\right)\right)_{\mathcal{U}} \right\Vert^q_{(F_\lambda)_\mathcal{U}}\right)^{1/q}
= \lim_{\mathcal{U}} \left(\sum_{j=1}^k  \left\Vert T_\lambda\left(x_\lambda^{1,j},\ldots,x_\lambda^{n,j}\right) \right\Vert^q_{F_\lambda}\right)^{1/q}\\
& ~~~~~\leq \lim_{\mathcal{U}} \left[ \|T_\lambda\|_{C_q^n} \cdot \prod_{i=1}^n \left(\int_0^1 \left\Vert \sum_{j=1}^k r_j(t)x_\lambda^{i,j}\right\Vert^2_{E_{i,\lambda}} dt\right)^{1/2}\right]\\
& ~~~~~\leq \sup_\lambda \|T_\lambda\|_{C_q^n} \cdot \prod_{i=1}^n \lim_{\mathcal{U}} \left(\int_0^1 \left\Vert \sum_{j=1}^k r_j(t)x_\lambda^{i,j}\right\Vert^2_{E_{i,\lambda}} dt\right)^{1/2}\\
& ~~~~~\overset{(*)}{=} \sup_\lambda \|T_\lambda\|_{C_q^n} \cdot \prod_{i=1}^n \left(\int_0^1 \left\Vert \sum_{j=1}^k r_j(t)\left(x_\lambda^{i,j}\right)_{\mathcal{U}}\right\Vert^2_{(E_{i,\lambda})_\mathcal{U}} dt\right)^{1/2},
\end{align*}
where in $(*)$ we use the fact that $\displaystyle\int_0^1 \left\Vert \sum_{j=1}^k r_j(t)x_\lambda^{i,j}\right\Vert^2_{E_{i,\lambda}} dt$ is a finite sum for every $i$. Thus
\[\overline{(T_\lambda)}_{\mathcal{U}} \in {\cal C}_q^n\left((E_{1,\lambda})_{\mathcal{U}},\ldots,(E_{n,\lambda})_{\mathcal{U}};(F_{\lambda})_
{\mathcal{U}}\right)\ \ \mathrm{and}\ \ \left\|\overline{(T_\lambda)}_{\mathcal{U}}\right\|_{{\cal C}_q^n} \leq \sup_\lambda\|T_\lambda\|_{{\cal C}_q^n
}.\]
\end{proof}

Now we can address the Aron-Berner stability of the ideals $\tau_{p_1,\ldots,p_n}$ and ${\cal C}_q^n$. Given an $n$-linear operator $A \in {\cal L}(E_1, \ldots, E_n;F)$, the Aron-Berner (or Arens or canonical) extension  $AB(A)$ of $A$ is an $n$-linear operator  $AB(A) \colon E_1'' \times \cdots \times E_n'' \longrightarrow F''$ defined in the following fashion: given $y_j'' \in E_j''$, $j = 1, \ldots,n $, define
\begin{equation}\label{iterated}AB(A)(y_1'', \ldots, y_n'') = \underset{\alpha_1}{w^*\!\!\mathrm{-}\!\lim} \cdots \underset{\alpha_n}{w^*\!\!\mathrm{-}\!\lim}\, (J_F\circ A)(x_{\alpha_1}, \ldots, x_{\alpha_n}),
\end{equation}
where, for $j = 1, \ldots, n$, $(x_{\alpha_j} )_{\alpha_j}$ is a net in $E_j$ such that the net $\left(J_{E_j}(x_{\alpha_j} )\right)_{\alpha_j}$ converges to $y_j''$ in the weak-star topology. $AB(A)$ is an extension of $A$ in the sense that
\begin{equation}AB(A) \circ (J_{E_1}, \ldots, J_{E_n}) = J_F \circ A \label{aronberner}
\end{equation}
(see \cite[Proposition 2.1]{AB}). Depending on the order the iterated limit is taken in (\ref{iterated}), $A$ can have several (at most $n!$) different Aron-Berner extensions. For more information on Aron-Berner extensions see, e.g., \cite{daviegamelin, dineen}. In the terminology of \cite{scand}, next result says that the ideals of multilinear operators of some proper type or of some proper cotype are Aron-Berner stable:

\begin{corollary} Let $(p_1, \ldots, p_n)$ be a proper type and $q$ be a proper cotype for $n$-linear operators. The following are equivalent for a multilinear operator $A \in {\cal L}(E_1, \ldots, E_n;F)$:\\
{\rm (a)} $A$ is of type $(p_1, \ldots, p_n)$ (of cotype $q$, respectively);\\
{\rm (b)} All Aron-Berner extensions of $A$ are of type $(p_1, \ldots, p_n)$ (of cotype $q$, respectively);\\
{\rm (c)} Some Aron-Berner extension of $A$ is of type $(p_1, \ldots, p_n)$ (of cotype $q$, respectively).
\end{corollary}

\begin{proof} As $J_E$ is a linear isometric embedding, (c) $\Longrightarrow$ (a)  follows easily from (\ref{aronberner}); and as (b) $\Longrightarrow$ (c) is obvious, we just have to prove (a) $\Longrightarrow$ (b). Suppose that $A$ has type $(p_1, \ldots, p_n)$ and that $AB(A)$ is an Aron-Berner extension of $A$. Let
$$AB(J_F \circ A)\colon E_1'' \times \cdots \times E_n'' \longrightarrow F^{(iv)}$$ be the Aron-Bener extension of $J_F \circ A$ with the same order of iterated limits as in $AB(A)$. Any dual space is canonically complemented in its bidual (Dixmier's Theorem), and we denote by $\pi$ the canonical projection from $F^{(iv)}$ to $F''$. By Theorem \ref{maxi} we know that the multi-ideal $\tau_{p_1,\ldots,p_n}$ is maximal, so from \cite[p.\,1750]{galicerjmaa} it follows that $AB(J_F \circ A)$ is of type $(p_1, \ldots, p_n)$ and that its range lies in $J_{F''}(F'')$. In particular, the diagram below is commutative:

\begin{equation*}\label{diagram}
\begin{gathered}
\xymatrix@C5pt@R15pt{
E_1 \ar@/_/[dd]_*{J_{E_1}} & \times & \cdots  & \times & E_n \ar@/_/[dd]_*{J_{E_n}} \ar[rrrrr]^*{A} & & & & & F \ar[rrrrr]^*{J_F} & & & & & F^{''} &\\
&  & \cdots &  &  & & & & &  & & & & &  &\\
E_1^{''} & \times & \cdots & \times & E_n^{''} \ar[rrrrrrrrrr]_*{AB(J_F \circ A)} \ar@/_/[uurrrrrrrrrr]^*{AB(A)} & & & & & & & & & & F^{(iv)} \ar@/_/[uu]_*{\pi} & \hspace{-0.6em}\supseteq J_{F^{''}}(F^{''}).
}
\end{gathered}
\end{equation*}
By the ideal property we have that $AB(A) = \pi \circ AB(J_F \circ A)$ is of type $(p_1, \ldots, p_n)$. The case of cotype is analogous.
\end{proof}

\noindent{\bf Acknowledgment.} The authors thank D. Galicer for his helpful suggestions and for drawing our attention to reference \cite{galicerjmaa}.

\vspace{1cm}

\noindent Faculdade de Matem\'atica~~~~~~~~~~~~~~~~~~~~~~Departamento de Ci\^{e}ncias Exatas\\
Universidade Federal de Uberl\^andia~~~~~~~~ Universidade Federal da Para\'iba\\
38.400-902 -- Uberl\^andia -- Brazil~~~~~~~~~~~~ 58.297-000 -- Rio Tinto -- Brazil\\
e-mail: botelho@ufu.br ~~~~~~~~~~~~~~~~~~~~~~~~~e-mails: jamilson@dcx.ufpb.br,\\ \hspace*{9,3cm}jamilsonrc@gmail.com
\end{document}